\documentclass{amsart}
\usepackage{amsmath, amssymb, amsthm}
\usepackage{mathtools}
\usepackage{amscd}
\usepackage{graphicx} 
\usepackage{enumerate}
\usepackage{todonotes}

\def\Qbar{\overline{\mathbb Q}}

\def\0{{\mathbf 0}}
\def\1{{\mathbf 1}}

\def\Kbar{{\bar K}}

\def\Aut{\mathrm{Aut}}

\def\Gal{\mathrm{Gal}}

\def\PrePer{\mathrm{PrePer}}

\usepackage[all]{xy}

\newtheorem{thm}{Theorem}

\newtheorem{prop}[thm]{Proposition}
\newtheorem{lemma}[thm]{Lemma}

\newtheorem{cor}[thm]{Corollary}
\newtheorem{question}[thm]{Question}

\newtheorem*{thm*}{Theorem}
\newtheorem*{alg*}{Algorithm}
\newtheorem*{lemma*}{Lemma}

\theoremstyle{remark}
\newtheorem{rmk}{Remark}
\newtheorem*{rmk*}{Remark}

\newtheorem*{notation*}{Notation}

\newtheorem*{example*}{Example}

\theoremstyle{definition}
\newtheorem{defn}[thm]{Definition}
\newtheorem*{defn*}{Definition}

\newcommand{\mybf}{\mathbb}

\newcommand{\bE}{\mybf{E}}

\newcommand{\bP}{\mybf{P}}
\newcommand{\bR}{\mybf{R}}
\newcommand{\bM}{\mybf{M}}

\newcommand{\bC}{\mybf{C}}

\newcommand{\bQ}{\mybf{Q}}

\newcommand{\bA}{\mybf{A}}

\newcommand{\cM}{\mathcal{M}}

\newcommand{\cX}{\mathcal{X}}
\newcommand{\al}{\alpha}

\providecommand{\abs}[1]{\lvert#1\rvert}
\providecommand{\norm}[1]{\lVert#1\rVert}

\newcommand{\ON}[1]{\operatorname{#1}}

\newcommand{\ra}{\rightarrow}

\newcommand{\ep}{\epsilon}

\newcommand{\p}{\partial}

\newcommand{\lfrac}[2]{\left(\frac{#1}{#2}\right)}

\newcommand{\Diag}{\mathrm{Diag}}

\newcommand{\logabs}[1]{\log\,\abs{#1}}

\newcommand{\dlim}{\mathop{\varinjlim}\limits}

\def\talltareesidedbox#1{\setbox0=\hbox{$#1$}\dimen0=\wd0 \advance\dimen0 by3pt\rlap{\hbox{\vrule height10pt width.4pt
 depth2pt \kern-.4pt\vrule height10.4pt width\dimen0 depth-10pt\kern-.4pt \vrule height10pt width.4pt depth2pt}}
 \relax \hbox to\dimen0{\hss$#1$\hss}}
\def\tareesidedbox#1{\setbox0=\hbox{$#1$}\dimen0=\wd0 \advance\dimen0 by3pt\rlap{\hbox{\vrule height8pt width.4pt
 depth2pt \kern-.4pt\vrule height8.4pt width\dimen0 depth-8pt\kern-.4pt \vrule height8pt width.4pt depth2pt}}
\relax \hbox to\dimen0{\hss$#1$\hss}}
\def\shorttareesidedbox#1{\setbox0=\hbox{$#1$}\dimen0=\wd0 \advance\dimen0 by3pt\rlap{\hbox{\vrule height7pt width.4pt
 depth2pt \kern-.4pt\vrule height7.4pt width\dimen0 depth-7pt\kern-.4pt \vrule height7pt width.4pt depth2pt}}
 \relax \hbox to\dimen0{\hss$#1$\hss}}

\newcommand{\dist}{\operatorname{dist}}

\newcommand{\sP}{\mathsf{P}}
\newcommand{\sA}{\mathsf{A}}
\newcommand{\lp}{(\!(\,}
\newcommand{\rp}{\, )\!)}

\title[Mutual energy and unlikely intersections]{A metric of mutual energy and unlikely intersections for dynamical systems}
 
\author[Fili]{Paul Fili}
\address{Department of Mathematics\\ Oklahoma State University, Stillwater, OK 74078}
\email{paul.fili@okstate.edu}
\thanks{The author would like to thank Igor Pritsker for many helpful conversations and for bringing \cite{Lowner} to his attention, and to thank Robert Rumely for suggestions which improved Proposition \ref{prop:main-ub}.}
\date{\today}

\usepackage[pdftitle={A metric of mutual energy and unlikely intersections},pdfauthor={Fili},pdfstartview={}]{hyperref}

\begin{document}

\begin{abstract}
 We introduce a metric of mutual energy for adelic measures associated to the Arakelov-Zhang pairing. Using this metric and potential theoretic techniques involving discrete approximations to energy integrals, we prove an effective bound on a problem of Baker and DeMarco on unlikely intersections of dynamical systems, specifically, for the set of complex parameters $c$ for which $z=0$ and $1$ are both preperiodic under iteration of $f_c(z)=z^2 + c$.
\end{abstract}

\maketitle

\allowdisplaybreaks[2]

\section{Introduction}

Petsche, Szpiro, and Tucker \cite{PST-pairing}, using local analytic machinery on Berkovich space, proved that the arithmetic intersection product introduced by Arakelov \cite{Arakelov} and studied by many others, particularly Zhang \cite{ZhangSmallPoints} in the dynamical context, which they termed the \emph{Arakelov-Zhang height pairing} in this context, satisfied the following theorem:
\begin{thm*}[Petsche, Szpiro, Tucker 2011]
If $\phi,\psi$ are rational maps of degree at least $2$ defined over a number field $K$, then the conditions:
\begin{enumerate}
  \item The Arakelov-Zhang height pairing $\langle \phi,\psi\rangle$ vanishes, 
  \item The associated dynamical Weil heights $h_\phi$ and $h_\psi$ are equal,
  \item The sets of preperiodic points $\PrePer(\phi)$ and $\PrePer(\psi)$ are equal.
  \item The intersection $\PrePer(\phi)\cap \PrePer(\psi)$ is infinite, and
  \item $\displaystyle \liminf_{\al\in \bP^1(\overline K)} (h_\phi(\al)+h_\psi(\al))=0$,
\end{enumerate}
are all equivalent, where $\PrePer(\phi)$ denotes the set of preperiodic points of $\phi$ and $h_\phi$ denotes the Call-Silverman canonical height \cite{CallSilverman}, and likewise for $\psi$.
\end{thm*}

The first goal of this note is to prove that in fact the Arakelov-Zhang pairing is the square of a metric on a certain space of adelic measures $\cM$. This is done by studying the Arakelov-Zhang pairing as a sort of mutual energy pairing in the spirit of Favre and Rivera-Letelier \cite{FRL,FRLcorrigendum}. As it requires some technical background, we will defer the precise definition of adelic measures and the Arakelov-Zhang pairing to Section \ref{sec:background} below. Our first result is the following:

\begin{thm}\label{thm:triangle-ineq}
 For any adelic measures $\rho,\sigma$, let $\langle \rho,\sigma\rangle$ denote the Arakelov-Zhang pairing, and
\[
 d(\rho,\sigma) = \langle \rho,\sigma\rangle^{1/2}.
\]
 Then $d$ is a metric on the real vector space of all adelic measures, that is, for all adelic measures $\rho,\sigma,\tau$, we have
\[
 d(\rho,\tau)\leq d(\rho,\sigma)+d(\sigma,\tau),
\]
 and $
 d(\rho,\sigma)=0$ if and only if $\rho=\sigma$ as adelic measures.
\end{thm}
For an adelic measure $\rho$ defined over $K$, we denote by $h_\rho$ the associated Weil height and
\[
 Z(\rho) = \{ \al\in\bP^1(\Kbar) : h_\rho(\al)\leq 0\}.
\]
\begin{thm}\label{thm:2}
Let $\rho,\sigma$ be adelic measures such that $Z(\rho)\cup Z(\sigma)$ is infinite. Then the following conditions are equivalent:
 \begin{enumerate}
  \item $d(\rho,\sigma)=0$,
  \item $\rho=\sigma$ as adelic measures,
  \item $h_\rho=h_\sigma$, 
  \item $Z(\rho)=Z(\sigma)$,
  \item $Z(\rho)\cap Z(\sigma)$ is infinite, and
  \item $\displaystyle \liminf_{\al\in \bP^1(\overline K)} (h_\rho(\al)+h_\sigma(\al))=0$,
 \end{enumerate}
\end{thm}
\begin{rmk}
The condition that $Z(\rho)\cup Z(\sigma)$ is infinite in the above theorem cannot be removed, as there are trivial examples of different adelic measures with $Z(\rho)\cup Z(\sigma)=\varnothing$ but $\rho\neq\sigma$, for example, we might take over $\bQ$ the measures $\rho,\sigma$ with $\rho_p = \sigma_p$ equal to the standard measure for $p<\infty$ (see Section \ref{sec:basic-potential-theory} below for the relevant definitions), $\rho_\infty$ the logarithmic equilibrium measure of $[-1,1]$, and $\sigma_\infty$ the logarithmic equilibrium measure of $[-1/2,1/2]$; both sets trivially must have $Z(\rho)=Z(\sigma)=\varnothing$ by the Fekete-Szeg\H{o} theorem as both sets have global capacity less than 1, but clearly $\rho\neq\sigma$. (It remains the case that $d(\rho,\sigma)>0$, as $\rho\neq\sigma$.) 
\end{rmk}

Part of the interest in the Arakelov-Zhang pairing lies in the theorem of Petsche, Szpiro and Tucker that if $\al_n\in\bP^1(\overline K)$ is a sequence of mutually distinct points such that
\begin{equation}\label{eqn:pst-heightzero}
 h_\phi(\al_n)\ra 0,\quad\text{then}\quad h_\psi(\al_n) \ra \langle \phi,\psi\rangle.
\end{equation}
The symmetry of the pairing $\langle \phi,\psi\rangle=\langle\psi,\phi\rangle$ reveals a remarkable symmetry in the limit above which na\"ively would have been far from obvious. As applications of their results, Petsche, Szpiro and Tucker come up with several explicit height difference bounds and formulas for the pairing in specific instances and establish some connections to special values of certain $L$-functions (see \cite[Prop. 18]{PST-pairing}).

In fact, once the metric property is recognized, the Arakelov-Zhang pairing becomes even more useful. The second goal of this paper is to give an application of this result to a problem of unlikely intersections in arithmetic dynamics that illustrates the utility of the triangle inequality for the mutual energy metric.

We recall a question posed by Umberto Zannier at the AIM workshop ``The uniform boundedness conjecture in arithmetic dynamics'' in 2008:
\begin{question}\label{question:zannier}
 Let $S_{0,1}$ denote the set of parameters $c\in \bC$ such that $z=0,1$ are both preperiodic under iteration of $f_c(z)=z^2+c$. Is $S_{0,1}$ finite?
\end{question}
\noindent Inspired by analogous problems in arithmetic geometry of recent interest, such questions are referred to as problems of \emph{unlikely intersections} in arithmetic dynamics. (We refer the interested reader to the recent book of Zannier \cite{ZannierBook} on the subject of unlikely intersection problems.) Baker and DeMarco \cite{BakerDeMarco} were able to answer this question in the affirmative using local analytic techniques involving equidistribution. Specifcally, Baker and DeMarco proved that $S_{0,1}$ is finite, and moreover that the more general set $S_{a,b}$ of parameters $c$ for which $z=a,b$ are both preperiodic under iteration of $f_c(z)=z^2+c$ is infinite if and only if $a^2=b^2$, in which case the generalized Mandelbrot sets $M_a$ and $M_b$ coincide (see Section \ref{sec:background} for further definitions).

However, these equidistribution results were ineffective and did not allow for explicit computation of the sets $S_{a,b}$. Nevertheless, based on numerical evidence, Baker and DeMarco conjectured that in fact the set $S_{0,1}=\{-2,-1,0\}$ (see \cite[Conj. 1.6]{BakerDeMarco}). Using the metric of mutual energy and some discrete approximation techniques, we are able to prove the following:
\begin{thm}\label{thm:degree-bound}
 Suppose that $c\in S_{0,1}$. Then $c$ is algebraic over the field of rational numbers of degree $[\bQ(c):\bQ]\leq 108$.
\end{thm}

Clearly, $S_{0,1}$ only contains algebraic integers, and in fact, it is easy to see that the height of such a set must be bounded, as such a $c$ must contain all of its Galois conjugates in $M_0$, which itself is contained in the disc $D(0,2)=\{z\in\bC : \abs{z}\leq 2\}$. This gives a bound on the height of such $c$ (trivially, $h(c)\leq \log 2$) and thus Theorem \ref{thm:degree-bound} gives an effective bound on the set of possible $c\in S_{0,1}$.

The main idea behind the proof of Theorem \ref{thm:degree-bound} is similar to that of \cite{BakerDeMarco}, namely, that probability measures equally supported on the Galois conjugates of such a number $c$ begin to equidistribute along the equilibrium measures of both $M_0$ and $M_1$, but as these measures are distinct, this cannot be done too closely. The new ideas introduced in this paper which allow us to make these bounds effective are the notion of distance between two (adelic) measures defined via the metric above, and in particular, the triangle inequality for this metric. Together with discrete energy approximation techniques, these ideas can be used to obtain quantitative bounds involving the degree of the algebraic number $c$ (see Propositions \ref{prop:main-ub} and \ref{prop:main-lb} below), leading to our results.

Lastly, we note that the technique used to prove the above result also can be used to give effective bounds for other unlikely intersection problems in arthmetic dynamics. The essential information needed to apply the techniques introduced in this paper are some potential theoretic information about the desired measures, particularly a bound on the modulus of continuity for the associated potential function near the boundary (which we do below in Section \ref{sec:mu0-mu1-ub}), and estimates on the mutual energy distance between the two measures, which can be obtained via finding algebraic numbers whose Galois conjugates roughly equidistribute along these measures (which we do below in Section \ref{sec:mu0-mu1-lb}). 

\section{Background and Notation}\label{sec:background}
\subsection{Basic potential theory}\label{sec:basic-potential-theory}
We will denote by $\bA^1,\bP^1$ the usual affine and projective lines and by $\sA^1,\sP^1$ the Berkovich affine and projective lines, respectively. We refer the reader to \cite{BakerRumelyBook,FRL,BakerBerkArticle} for some basic references on Berkovich space. We define the \emph{standard measures} $\lambda_v$ on $\sP^1(\bC_v)$ to be the probability measures which are either the Dirac measure on the Gauss point of $\sP^1(\bC_v)$ if $v\nmid\infty$ or the normalized Haar measure on the unit circle of $\bC^\times$ if $v\mid\infty$. We let $\Delta$ denote the measure-valued Laplacian on $\sP^1$. We recall the following definition from \cite{FRL}:
\begin{defn}\label{defn:adelic-measure}
 Let $K$ be a number field. We call $\rho=(\rho_v)_{v\in M_K}$ an \emph{adelic measure} if for each $v\in M_K$, $\rho_v$ is a Borel probability measure on $\sP^1(\bC_v)$ which is equal to $\lambda_v$ for all but finitely many $v$ and admits a continuous potential with respect to $\lambda_v$ at the remaining places, that is, for which
 $
  \rho_v-\lambda_v = \Delta g
 $
 for some $g\in C(\sP^1(\bC_v))$.\footnote{Adelic measures are defined twice in \cite{FRL}, first in Defn. 1.1 and later in Defn. 5.1. The condition on the exceptional archimedean places is defined slightly differently in Defn. 5.1 as merely admitting a continuous potential $\rho_v=\Delta u$ locally in some neighborhood of each point on $\sP^1(\bC_v)=\bP^1(\bC_v)$. However one can check easily that the conditions of locally admitting a continuous potential and having a global continuous potential with respect to $\lambda_v$ are in fact equivalent, as any signed Borel measure with total measure zero is the Laplacian of some locally integrable function on $\bP^1(\bC)$, and the local condition implies that this global function must be continuous too.}
\end{defn}

Let $\cM$ denote the space of all adelic measures as defined above. Notice that if $\cM_K$ denotes the set of adelic measures defined over $K$ then, under the natural inclusion maps, we can view our space of all adelic measures as the direct limit over all number fields $K/\bQ$: 
\[
 \cM = \dlim_{K/\bQ} \cM_K.
\]
In fact, we can make an even more advantageous decomposition in the following fashion: Let $\cX$ denote the real vector space of all signed Borel measures spanned by the span of set
 \[\{ \rho-\sigma : \rho,\sigma\in \cM_K\},\]
and for each place $v$ of a number field $K$, let $\cX_v$ denote the real vector space generated by the set of differences of $v$-adic adelic measures $\rho_v-\sigma_v$. We note that if $\rho$ is defined over $K$ and $\sigma$ over $L$ then $\rho-\sigma$ is defined over the compositum $KL$, and hence each element of $\cX$ is defined over some number field, and thus can naturally think of $\cX$ as a direct limit
 \[
  \cX = \dlim_{K} \bigoplus_{v\in M_K} \cX_v,
 \]
 where our number fields are partially ordered by inclusion. Note that by our assumptions for adelic measures, for any $\rho-\sigma$ we will have $\rho_v-\sigma_v = \lambda_v-\lambda_v = 0$ for almost all $v$.

If $\rho$ is an adelic measure over $K$ then we can define canonical heights $h_\rho$ on $\overline{K}$ which satisfy an equidistribution theorem (see \cite{FRL}; for dynamical heights an independent proof is given by \cite{BakerRumely}, and for a more geometric approach see for example \cite{ChambertLoirEqui}; these equidistribution generalize earlier work of \cite{SUZ,Bilu}).

\subsection{Mutual energy}
We fix a number field $K$ over whose completions our measures will be defined, and choose absolute values $\abs{\cdot}_v$ extending the usual absolute values on $\bQ$ and let $\norm{\cdot}_v = \abs{\cdot}_v^{[K_v:\bQ_v]/[K:\bQ]}$, so that the set of absolute values $\{ \norm{\cdot}_v : v\in M_K\}$ satisfies the product formula.

We assume that $\rho_v,\sigma_v$ are signed finite Borel measures on $\sP^1(\bC_v)$. We define, when it exists, the \emph{local mutual energy pairing} to be
\begin{equation}
 \lp\rho_v,\sigma_v\rp_v = {\iint}_{\sA^1_v\times\sA^1_v\setminus\Diag_v} -\log \norm{x-y}_v\,d\rho_v(x)\,d\sigma_v(y)
\end{equation}
 where $\sA^1_v=\sA^1(\bC_v)$ denotes the Berkovich affine line over $\bC_v$ and $\Diag_v = \{ (x,x) : x\in\bC_v\}$ (note we are only excluding the classical points of the diagonal). Throughout, in the non-archimedean case on the Berkovich line, the kernel $\norm{x-y}_v$ in the above integral should be read as the natural extension to the Berkovich line of this distance, which (up to normalization of the absolute value) is denoted by $\sup\{x,y\}$ in the article of Favre and Rivera-Letelier \cite[\S 3.3]{FRL} and as the \emph{Hsia kernel} $\delta(x,y)_\infty$ in the book of Baker and Rumely \cite[\S 4]{BakerRumelyBook}.
 
 When $\rho=(\rho_v),\sigma=(\sigma_v)$ are adelic measures we will sometimes write $\lp\rho,\sigma\rp_v$ instead of $\lp\rho_v,\sigma_v\rp_v$ to ease notation. When well-defined it is easy to see that the local mutual energy is symmetric. The local mutual energy exists in particular when $\rho_v$ and $\sigma_v$ are either Borel probability measures of continuous potentials with respect to the standard measure, that is, $\rho_v-\lambda_v=\Delta g$ for some $g\in C(\sP^1(\bC_v))$, or are probability measures supported on a finite subset of $\bP^1(\overline K)$. In particular this applies for our adelic measures, and extends naturally by bilinearity to the vector space of signed measures arising from these measures. We refer the reader to \cite{FRL} for proofs of these results.
 
 We define the \emph{mutual energy pairing} as the sum of the local mutual energies:
\begin{equation}
 \lp\rho,\sigma\rp = \sum_{v\in M_K} \lp \rho,\sigma\rp_v.
\end{equation}
Notice that the above sum is in fact finite for the measures under consideration: if $\rho,\sigma$ are adelic measures, then $\rho_v=\lambda_v=\sigma_v$ at all but finitely many places and $\lp\lambda_v,\lambda_v\rp_v=0$ for each $v\in M_K$, and if either is a probability measure with support a finite subset of $\bP^1(\overline K)$ almost all valuations will be trivial as usual (whether it is paired with an adelic measure or another such probability measure). Further, the choice of normalization for our $v$-adic absolute values ensures that the value above is well-defined under extension of our ground field, so it is an absolute quantity which does not depend on the particular choice of base field.

For $\al\in\bP^1(\overline{K})$, let $[\al]$ denote the probability measure supported equally on the Galois conjugates of $\al$ over $K$, that is,
\begin{equation}
 [\al]_K = [\al] = \frac{1}{\# G_K\al} \sum_{z\in G_K \al} \delta_z
\end{equation}
where $G_K=\Gal(\overline{K}/K)$ and $\delta_z$ denotes the Dirac measure at $z$, which for $z\in\bP^1(\overline{K})$ we interpret as the adelic measure of the point mass at $z$ at each place. (When the base number field $K$ is understood, we may drop the subscript on $[\al]_K$.) Then the canonical height $h_\rho:\bP^1(\overline K)\ra \bR$ associated to $\rho$ is defined to be 
\begin{equation}
 h_\rho(\al) = \frac{1}{2} \lp \rho-[\al], \rho-[\al]\rp.
\end{equation}
One can check that for the standard measure, $\lp\lambda,\lambda\rp=\lp[\al],[\al]\rp=0$ and $\lp\lambda,[\al]\rp_v = \log^+ \norm{\al}_v$, and so $h_\lambda=h$ coincides with the usual absolute logarithmic Weil height. When $\rho=(\rho_{\phi,v})_{v\in M_K}$ is the adelic set of canonical measures associated to iteration of a rational map $\phi$, then by \cite[Thm. 4]{FRL} we have $h_\rho=h_\phi$ is the usual Call-Silverman dynamical height \cite{CallSilverman}.

For an adelic measure $\rho$ defined over $K$ we define the set
\begin{equation}
 Z(\rho) = \{ \al\in \bP^1(\Kbar) = h_\rho(\al) \leq 0 \}.
\end{equation} 
Notice that if if $\rho$ is the canonical adelic measure associated to a rational map of degree at least $2$, then in fact $Z(\rho)$ is precisely the set of preperiodic points. 

\subsection{Generalized Mandelbrot sets}
Let us recall some of the notation we will use from \cite{BakerDeMarco}. Let $f_c(z) = z^2 + c$ for a (usually complex) parameter $c$. The Mandelbrot set $M=M_0$ is defined as
\begin{equation}
 M_{0} = \{ c\in\bC : \sup_n \abs{f^n_c(0)}< \infty\}
\end{equation}
where $f^n_c$ denotes the $n$th iterate of $f_c$. We can define an analogous set for different inital values; in particular, we let
\begin{equation}
 M_1 = \{ c\in\bC : \sup_n \abs{f^n_c(1)}< \infty\}.
\end{equation}
Both sets $M_0$ and $M_1$ are compact in $\bC$ with connected complements and logarithmic capacity $1$. We refer the reader to \cite[Proposition 3.3]{BakerDeMarco} for proofs and more details regarding the generalized Mandelbrot sets. We denote by $\mu_0$ the equilibrium measure of the set $M_0$ in the sense of complex potential theory, and by $\mu_1$ the equilibrium measure of the set $M_1$.

It is worth noting that we can (and should) view the above measures as the archimedean components of adelic measures $\boldsymbol{\mu}_0 = (\mu_{0,p})_{p\in M_\bQ}$ and $\boldsymbol{\mu}_1 = (\mu_{1,p})_{p\in M_\bQ}$  defined over $\bQ$ with $\mu_{0,\infty} = \mu_0$ and $\mu_{0,p}=\lambda_p$ for $p\nmid \infty$, and likewise for $\boldsymbol{\mu}_1$. These are equilibrium measures for the adelic sets $\mathbb{M}_0$ and $\mathbb{M}_1$, which consist of $M_0$ and $M_1$ at the archimedean prime, respectively, and the Berkovich unit disc at each finite prime. Baker and DeMarco thus construct a canonical height $h_{\mathbb{M}_0}$ and $h_{\mathbb{M}_1}$ relative to each set. We refer the reader to \cite{BakerDeMarco} for more details on this construction. As the non-archimedean components of these measures are trivial, we only require an analysis at the archimedean place for our desired application. Thus we will write $d(\mu_0,\mu_1)$ below where we might otherwise write $d(\boldsymbol{\mu}_0, \boldsymbol{\mu}_1)$, etc.

\subsection{Regularization of measures}\label{sec:meas-reg}
For this section (and particularly in Section \ref{sec:application}), we will now assume our base number field $K=\bQ$. As above for $z\in\bC$, $\delta_z$ denote the Dirac point mass at $z$. To each algebraic number $\alpha\in\Qbar$ of degree $d=[\bQ(\alpha):\bQ]$ we denote by 
\begin{equation}
[\al] = \frac{1}{d} \sum_{z\in G_\bQ \al} \delta_z 
\end{equation}
the probability measure on $\bC$ supported equally on each Galois conjugate of $\al$. We wish to find a regularization of $[\al]$ supported on $\bC$ which admits a continuous potential. To do this, we use a regularization introduced in \cite{F-P-quant} (which itself is quite similar to the technique used in \cite[\S 2]{FRL}).  

Specifically, for a given Dirac point mass $\delta_x$ for $x\in \bR$, we will define $\delta_{x,\ep}$ to be the normalized unit Lebesgue measure of the circle $\{ z\in\bC : \abs{z-x} = \ep\}$. We then define, for $\alpha$ an algebraic number of degree $d$,
\begin{equation}
 [\al]_\ep = \frac{1}{d} \sum_{z\in G_\bQ\alpha} \delta_{z,\ep}.
\end{equation}
It is immediate that these measures admit a continuous potential as defined above.

\section{The metric of mutual energy}
For adelic measures $\rho,\sigma$ we define the \emph{Arakelov-Zhang height pairing} to be
\begin{equation}
\langle \rho,\sigma\rangle = \frac{1}{2} \lp \rho-\sigma, \rho-\sigma\rp.
\end{equation}
It seems sensible to also suggest that we might use the notation $h_\rho(\sigma)$ in analogy to the definition of $h_\rho(\al)$, however, we note that we insist on averaging $\al$ over its Galois conjugates, and we make no such requirement on $\sigma$ (indeed, we have imposed no requirement that our adelic measures are $\Aut(\bC_v/K_v)$-stable at each place, though there are reasons we may want to restrict to such measures as a more interesting class; see Remark \ref{rmk:Zfinite} for more on this). Nevertheless, this similarity suggests that in fact the symmetry in \eqref{eqn:pst-heightzero} is quite natural. We note that if we define a \emph{local Arakelov-Zhang pairing} by
\begin{equation*}
 \langle \rho,\sigma\rangle_v = \frac{1}{2}\lp \rho-\sigma,\rho-\sigma\rp_v
\end{equation*}
then in fact the Arakelov-Zhang pairing can be expressed as
\begin{equation}
\langle \rho,\sigma\rangle = \sum_{v\in M_K} \langle \rho,\sigma\rangle_v.
\end{equation}

We will show that our pairing agrees with that defined in \cite{PST-pairing} when $\rho,\sigma$ are the canonical measures associated to iteration of rational maps of degree at least $2$. It is obvious from the definition that the height pairing satisfies
\[
 \langle \rho,\sigma\rangle = \langle \sigma,\rho\rangle\quad\text{and}\quad \langle\rho,\rho\rangle = 0.
\]

While for general signed Borel measures $\mu$ we may have $\lp \mu,\mu\rp_v<0$ for some $v$ (see \cite[\S 6]{FRL}), Favre and Rivera-Letelier show that if $\mu(\sP^1(\bC_v))=0$ and $\mu$ has a continuous potential, then in fact, $\lp \mu_v,\mu_v\rp_v\geq 0$, with equality if and only if $\mu_v=0$. It follows that if $\rho,\sigma$ are two adelic heights then $\mu=\rho-\sigma$ meets this criterion at each place, and thus locally we have $\langle \rho,\sigma\rangle_v\geq 0$ with equality if and only if $\rho_v=\sigma_v$, and so in fact we have:
\begin{prop}
Let $\rho,\sigma$ be adelic measures. Then $\langle \rho,\sigma\rangle \geq 0$, with equality if and only if $\rho=\sigma$.
\end{prop}

We also note as an almost immediate consequence of our definition:
\begin{prop}
 Let $\rho,\sigma$ be adelic measures. Then
 \begin{equation}
 \langle\rho,\sigma\rangle = h_\rho(\infty) + h_\sigma(\infty) + \sum_{v\in M_K} \iint_{\sA^1(\bC_v)\times \sA^1(\bC_v)} \log\norm{x-y}_v \,d\rho_v(x)\,d\sigma_v(y).
 \end{equation}
\end{prop}
\begin{proof}
 We merely note, as is easy to check from our definitions, that
 \[
  h_\rho(\infty)=\frac{1}{2}\lp \rho-[\infty],\rho-[\infty]\rp=\frac{1}{2}\lp \rho,\rho\rp,
 \]
 and likewise $h_\sigma(\infty)=\frac{1}{2}\lp \sigma,\sigma\rp$, and further that since $\rho,\sigma$ are adelic measures, the set $\Diag_v$ is of $\rho_v\otimes\sigma_v$-measure zero, so the result follows.
\end{proof}
As a corollary of this result, we recover \cite[Prop. 16]{PST-pairing} for adelic measures:
\begin{cor}
 Let $\lambda$ denote the standard adelic measure, which is the Dirac measure at the Gauss point at the finite places and the normalized Haar measure of the unit circle in the complex plane at the infinite places. Let $\rho$ be an adelic measure. Then
 \begin{equation}
  \langle \rho,\lambda\rangle = h_\rho(\infty) + \sum_{v\in M_K} \int_{\sA^1(\bC_v)} \log^+ \norm{x}_v\,d\rho_v(x).
 \end{equation}
\end{cor}
\begin{proof}
 Note that $h_\lambda(\infty)=h(\infty)=0$ for the standard height, and that the $v$-adic potential function is $\log^+\norm{x}_v$ for the standard measure.
\end{proof}

We now prove the key result of this paper, which by the results of Petsche, Szpiro, and Tucker is enough to conclude that our pairing agrees with the Arakelov-Zhang pairing when $\rho,\sigma$ arise from iterating rational maps of degree at least $2$:
\begin{thm}\label{thm:limit-thm}
 Let $\rho,\sigma$ be adelic measures defined over the number field $K$. If $\{\al_n\}\subset \bP^1(\overline{K})$ is a sequence of mutually distinct algebraic numbers, then
 \[
  h_\sigma(\al_n) \ra \langle \rho,\sigma\rangle\quad\text{whenever}\quad h_\rho(\al_n) \ra 0.
 \]
\end{thm}
\begin{proof}
 As above we let $[\al_n]$ denote the probability measure supported equally on the $K$-Galois conjugates of $\al_n$, and as above we think of $[\al_n]$ as an adelic measure.  Let us compute $h_\sigma(\al_n)$:
 \begin{equation*}
 \begin{aligned}
  h_\sigma(\al_n) &= \frac{1}{2} \sum_{v\in M_K} \lp \sigma-[\al_n], \sigma-[\al_n]\rp_v\\
  &= \frac{1}{2} \sum_{v\in M_K} \left[ \lp \sigma,\sigma\rp_v -2\lp \sigma,[\al_n]\rp_v + \lp [\al_n],[\al_n]\rp_v\right]
 \end{aligned}
 \end{equation*}
 We will first analyze the middle term, which is of the most interest. Write
 \[
  -2\lp \sigma,[\al_n]\rp_v = -2\lp \sigma-\rho,[\al_n]\rp_v + 2\lp \rho,[\al_n]\rp_v.
 \]
 The condition that $h_\rho(\al_n)\ra 0$ implies by the equidistribution theorem \cite[Thm. 2]{FRL} that we have weak convergence of measures
 \[
  [\al_n] \ra \rho_v\quad\text{at each place }v.
 \]
 By our assumptions on adelic heights, we know there exists a continuous $f_v\in C(\sP^1(\bC_v))$ such that
 $
  \rho_v - \sigma_v = \Delta f_v.
 $
 Let $g_v(x) = \int_{\sA^1(\bC_v)} -\log \abs{x-y}_v\,d(\sigma_v-\rho_v)$. By Lemmas 2.5 and 4.4 of \cite{FRL}, we know that $g_v$ is integrable with respect to $[\al_n]$. Since $\sigma_v-\rho_v$ is an adelic measure, it does not charge the point $\infty\in\sP^1(\bC_v)$, so $\Delta g_v=\sigma_v-\rho_v$ (see for example \cite[Ex. 5.17]{BakerRumelyBook}) and thus we must have that $f_v-g_v$ is constant, so in fact $g_v$ is continuous everywhere as well. If we let $G_v(x)=\frac{[K_v:\bQ_v]}{[K:\bQ]} g_v(x)$, then we can conclude that
 \begin{align*}
  \lim_{n\ra\infty}\lp \sigma-\rho,[\al_n]\rp_v
  &= \lim_{n\ra\infty}\int_{\sA^1_v\times \sA^1_v\setminus\Diag_v} -\log\norm{x-y}_v\,d(\sigma_v-\rho_v)(x)\,d[\al_n](y)\\
  &= \lim_{n\ra\infty}\int_{\sA^1(\bC_v)} G_v(x)\,d[\al_n](x)\\
  &= \lim_{n\ra\infty}\int_{\sP^1(\bC_v)} G_v(x)\,d[\al_n](x),
\end{align*}
 where the last equality follows since we may as well assume $\al_n\neq \infty$ for all $n$, and we can apply weak convergence of measures $[\al_n]\ra \rho_v$ on $\sP^1(\bC_v)$ to conclude
\begin{align*}
  \lim_{n\ra\infty}\lp \sigma-\rho,[\al_n]\rp_v&=\int_{\sP^1(\bC_v)} G_v(x)\,d\rho_v(x)
  = \int_{\sA^1(\bC_v)} G_v(x)\,d\rho_v(x)\\
  &= \int_{\sA^1\times \sA^1\setminus\Diag} -\log\norm{x-y}_v\,d(\sigma_v-\rho_v)(y)\,d\rho_v(x)\\
  &= \lp \sigma-\rho,\rho\rp_v.
 \end{align*}
 So we see that 
 \[
  \lp \sigma-\rho,[\al_n]\rp_v\ra \lp \sigma-\rho,\rho\rp_v\quad\text{for each }v\in M_K.
 \]
 Now, by our assumption that $\rho,\sigma$ are adelic measures, we have $\rho_v=\sigma_v=\lambda_v$ at all but finitely many places $v$, independent of $n$, so in fact almost all terms are zero independent of $n$, so we can say that
 \begin{equation}\label{eqn:sigma-aln}
 \sum_{v\in M_K}\lp \sigma-\rho,[\al_n]\rp_v\ra \sum_{v\in M_K} \lp \sigma-\rho,\rho\rp_v\quad\text{as}\quad n\ra\infty.
 \end{equation}

 Now,
\[
 h_\rho(\al_n) = \frac{1}{2} \sum_{v\in M_K} \left[ \lp \rho,\rho\rp_v -2\lp \rho,[\al_n]\rp_v + \lp [\al_n],[\al_n]\rp_v\right]\ra 0
\]
by assumption. Notice that
\[
 \lp [\al_n],[\al_n]\rp_v = \frac{1}{[K(\al_n):K]^2} \log \bigg|\prod_{\substack{\beta,\gamma\in G_K\al_n\\ \beta\neq\gamma}} (\beta-\gamma)\bigg|_v,
\]
so by the product formula,
\begin{equation}\label{eqn:disc}
\sum_v \lp [\al_n],[\al_n]\rp_v =0,
\end{equation}
and thus we can conclude that
\begin{equation}\label{eqn:rho-aln}
 \sum_v 2\lp \rho,[\al_n]\rp_v\ra \sum_v \lp \rho,\rho\rp_v.
\end{equation}

Combining \eqref{eqn:sigma-aln}, \eqref{eqn:disc}, and \eqref{eqn:rho-aln}, we obtain the desired result:
\begin{multline*}
 \lim_{n\ra\infty} h_\sigma(\al_n) = \lim_{n\ra\infty} \frac{1}{2} \sum_{v\in M_K} \left[ \lp \sigma,\sigma\rp_v -2\lp \sigma,[\al_n]\rp_v + \lp [\al_n],[\al_n]\rp_v\right]\\
 = \lim_{n\ra\infty} \frac{1}{2} \bigg( \sum_{v\in M_K} \lp \sigma,\sigma\rp_v -\sum_{v\in M_K} 2\lp \sigma-\rho,[\al_n]\rp_v - \sum_{v\in M_K} 2 \lp \rho,[\al_n]\rp_v \\
 + \sum_{v\in M_K} \lp [\al_n],[\al_n]\rp_v\bigg)\\
 = \frac{1}{2} \sum_{v\in M_K} \left[ \lp \sigma,\sigma\rp_v -2\lp \sigma-\rho,\rho\rp_v - \lp \rho,\rho\rp_v\right]\\
 = \frac{1}{2}\sum_{v\in M_K} \lp \sigma-\rho,\sigma-\rho\rp_v
 = \frac{1}{2}\lp \sigma-\rho,\sigma-\rho\rp = \langle \sigma,\rho\rangle.\qedhere
\end{multline*} 
\end{proof}

\subsection{The mutual energy metric}
Let $\cX$ denote the real vector space of all signed Borel measures spanned by the span of set \[\{ \rho-\sigma : \rho,\sigma\text{ are adelic measures}\},\] and for each place $v$ of a number field $K$, let $\cX_v$ denote the real vector space generated by the set of differences of $v$-adic adelic measures $\rho_v-\sigma_v$. We note that if $\rho$ is defined over $K$ and $\sigma$ over $L$ then $\rho-\sigma$ is defined over the compositum $KL$, and hence each element of $\cX$ is defined over some number field, and thus can naturally think of $\cX$ as a direct limit
 \[
  \cX = \dlim_{K} \bigoplus_{v\in M_K} \cX_v,
 \]
 where our number fields are partially ordered by inclusion. Note the direct sum is used here as it easy to see that by our assumptions for adelic measures, for any $\rho-\sigma$ we will have $\rho_v-\sigma_v = \lambda_v-\lambda_v = 0$ for almost all $v$.

Using this result, we are now in a position to prove our main theorem.
\begin{proof}[Proof of Theorem \ref{thm:triangle-ineq}]
 First, we prove that $d(\rho,\sigma)$ satisfies the triangle inequality on the space of adelic measures. To see this, let $\cX$ denote the real vector space of all signed Borel measures spanned by the span of set \[\{ \rho-\sigma : \rho,\sigma\text{ are adelic measures}\},\] and for each place $v$ of a number field $K$, let $\cX_v$ denote the real vector space generated by the set of differences of $v$-adic adelic measures $\rho_v-\sigma_v$. 
 
 It follows from Propositions 2.6 and 4.5 of \cite{FRL} that at each place $v$ the energy pairing $\lp \cdot,\cdot\rp_v$ is a symmetric, positive definite bilinear form on $\cX_v$. In particular, we can define $\norm{\mu}_v = \lp \mu,\mu\rp_v^{1/2}$ for $\mu\in \cX_v$ and by the usual arguments this defines a the vector space norm on $\cX_v$. Then the Arakelov-Zhang pairing is equal to, for $\rho,\sigma$ adelic measures defined over $K$, 
\[
 \langle \rho,\sigma\rangle = \frac{1}{2}\sum_{v\in M_K} \lp \rho_v-\sigma_v,\rho_v-\sigma_v  \rp_v = \frac{1}{2}\sum_{v\in M_K} \norm{\rho_v-\sigma_v}_v^2.
\]
(Notice that $\rho_v-\sigma_v\in \cX$.) It follows from the positive-definiteness of the norms $\norm{\cdot}_v$ that $d(\rho,\sigma)$ will vanish if and only if $\rho_v=\sigma_v$ at every place, or equivalently, if $\rho=\sigma$. Further,
\[
 d(\rho,\sigma) = \langle \rho,\sigma\rangle^{1/2} = \bigg( \frac{1}{2} \sum_{v\in M_K} \norm{\rho_v-\sigma_v}^2 \bigg)^{1/2}
\]
satisfies the triangle inequality, using the triangle inequality for $\norm{\cdot}_v$ at each place $v$ and the usual $\ell^2$-triangle inequality for the entire sum; specifically, suppose $\rho,\sigma,\tau$ are all adelic measures over $K$, then we have
\begin{align*}
 d(\rho,\tau) &= \bigg( \frac{1}{2}\sum_{v\in M_K} \norm{\rho_v-\sigma_v - (\tau_v-\sigma_v)}^2 \bigg)^{1/2}\\
&\leq \bigg( \frac{1}{2}\sum_{v\in M_K} (\norm{\rho_v-\sigma_v}_v + \norm{\sigma_v-\tau_v}_v)^2 \bigg)^{1/2}\\
&\leq \bigg( \frac{1}{2}\sum_{v\in M_K} \norm{\rho_v-\sigma_v}_v^2\bigg)^{1/2} + \bigg( \frac{1}{2}\sum_{v\in M_K}  \norm{\sigma_v-\tau_v}_v^2 \bigg)^{1/2}\\
&= d(\rho,\sigma)+d(\sigma,\tau).
\end{align*}
We note in passing that all of the above sums are in fact finite, as by definition $\sigma_v = \rho_v=\tau_v=\lambda_v$ is the standard measure at all but finitely many places.
\end{proof}
\begin{rmk}\label{rmk:local-triangle-ineq}
 We note that if we fix place $v$ of a number field $K$ and define a \emph{local} mutual energy $d(\rho_v,\sigma_v) = \langle \rho_v, \sigma_v\rangle^{1/2}$, then the proof in the above theorem also shows that $d$ is a metric on the local space of adelic metrics $\cX_v$.
\end{rmk}

\begin{proof}[Proof of Theorem \ref{thm:2}]
Since the mutual energy pairing is nondegenerate, (1) is equivalent to (2), and (3) follows immediately from (2). Clearly (3) $\Rightarrow$ (4), and (4) $\Rightarrow$ (5) since we assumed that $Z(\rho)\cup Z(\sigma)$ was infinite, and (5) $\Rightarrow$ (6) follows from the essential nonnegativity of the adelic height \cite[Thm. 6]{FRL}. It remains to show that (6) now implies (1), but it follows immediately by the equidistribution theorem of \cite{FRL} that if there is a sequence of algebraic numbers $\al_n$ which is simultaneously small for $\rho$ and $\sigma$, then at each place $v$ we have weak convergence $[\al_n]\ra \rho_v$ and $[\al_n]\ra \sigma_v$, but this implies that $\rho_v,\sigma_v$ define the same linear functionals on $C(\sP^1(\bC_v))$, and hence they are the same measures.
\end{proof}
\begin{rmk}\label{rmk:Zfinite}
 The assumption that $\rho$ or $\sigma$ is a global adelic measure, that is, that $Z(\rho)\cup Z(\sigma)$ is infinite, is essential to second part of this result. Dynamical heights, arising from iteration of rational maps, are always nonnegative and have an infinite set of preperiodic points. Both of these properties may fail for more general adelic measures, although adelic heights are always essentially nonnegative in the sense that
 \[
  \{ \al\in \bP^1(\overline{K}) : h_\rho(\al)<-\ep<0 \}
 \]
 is finite for any fixed $\ep>0$. But more important, as noted too in \cite[Rmk. 2.11]{BakerDeMarco}, it may happen that $Z(\rho)$ is finite; choose for example a compact Berkovich adelic set $\bE$ which avoids infinity of logarithmic capacity $\gamma_\infty(\bE)<1$, then one can show by the Fekete-Szeg\H o theorem \cite{BakerRumelyBook} that in fact $Z(\rho)$ must be finite.
 
 Further, while (1) is always equivalent to $\rho=\sigma$ for adelic measures, we cannot even assume that $\rho=\sigma$ and $h_\rho=h_\sigma$ are equivalent when $Z(\rho)\cup Z(\sigma)$ is finite. To see this, suppose that we took $\rho_p=\sigma_p=\lambda_p$ for all finite rational primes $p$, but $\rho_\infty$ to be the equilibrium measure of the line segment $\{ x + i\in\bC : -2\leq x\leq 2\}$ and $\sigma_\infty$ that of $\{ x - i\in\bC : -2\leq x\leq 2\}$. Then since we define the height $h_\rho(\al)$ for $\al\in\bP^1(\overline K)$ as an average over the Galois conjugates of $\al$, $h_\rho=h_\sigma$ even though $\rho\neq \sigma$. However, as we noted, exceptions like this can only occur when $Z(\rho)\cup Z(\sigma)$ is finite, which in this example is true because it is impossible to have a sequence of algebraic numbers equidistributing along either set in $\bC$ due to the failure of the measures to be stable under complex conjugation.
 \end{rmk}

\section{Application to the unlikely intersection problem}\label{sec:application}
Recall that by $S_{0,1}$ we denote the set of parameters $c\in \bC$ such that $z=0,1$ are both preperiodic under iteration of $f_c(z)=z^2+c$. In this section we prove Theorem \ref{thm:degree-bound}
. In order to prove these results, we will need to establish upper and lower bounds on the mutual energy of the equilibrium measures of two sets $M_0$ and $M_1$.

In this section we will only be performing our analysis at the archimedean prime, so we will write $(\cdot, \cdot)$ for $(\cdot,\cdot)_\infty$ for the archimedean energy pairing throughout this section. We will also set as our notation:
\[
 d_\infty(\rho,\sigma) = (\rho-\sigma, \rho-\sigma)^{1/2}
\]
dropping the factor of $1/2$ from the definition of the metric above for convenience. Notice that $d_\infty$ again defines a metric when applied to the space of Borel probability measures on $\bC$ that admit a continuous potential. Lastly, we note that for the adelic measures $\boldsymbol{\mu}_0$ associated to the adelic Mandelbrot set $\mathbb{M}_0$  and $\boldsymbol{\mu}_1$ associated to the adelic Mandelbrot set $\mathbb{M}_1$, that 
\[
 d_\infty(\mu_0,\mu_1) = \frac{1}{\sqrt{2}} d(\boldsymbol{\mu}_0,\boldsymbol{\mu}_1),
\] 
where $\mu_0,\mu_1$ are the archimedean components of $\boldsymbol{\mu}_0,\boldsymbol{\mu}_1$ respectively. 

\subsection{Upper bound on the mutual energy of $M_0$ and $M_1$}\label{sec:mu0-mu1-ub}
In this section, we will prove an upper bound (Proposition \ref{prop:main-ub}) on the mutual energy of the equilibrium measures of $M_0$ and $M_1$ based on the highest possible degree of a parameter $c$ for which $0$ and $1$ are both preperiodic for $f_c(z) = z^2 + c$. We begin with some preliminary lemmas.
\begin{lemma}\label{lemma:loewner}
 Let $K\subset\bC$ be a compact set of capacity of $1$ with connected complement $\Omega = \overline\bC\setminus K$. Then the Green's function $g(z)=g_\Omega(z,\infty)$ with respect to infinity satisfies:
 \[
  g(z)\leq \dist(z,K)^{1/2}\quad\text{for}\quad z\in \bC\setminus K,
 \]
 where $\dist$ denotes the usual Euclidean distance in $\bC$.
\end{lemma}
\begin{proof}
 Our proof relies on a result of L\"owner \cite{Lowner}, which states that for a continuum $K$ with connected complement $\Omega = \overline{\bC}\setminus K$, if $\phi : \Omega \ra \Delta = \{ w\in\overline \bC : \abs{w} > 1 \}$ is the standard conformal map satisfying $\phi(\infty)=\infty$ and $\phi'(\infty)>0$, then for all $u>0$,
 \[
  \dist(L_u,K) \geq \frac{u^2}{1+u} \ON{cap}(K),
 \]
 where $L_u = \{z\in\bC : \abs{\phi(z)}=1+u\}$. In particular, it follows in the case of $\ON{cap}(K)=1$ that for $z\in\Omega$ and $t\in \p \Omega = \p K$, we have that
 \[
  \abs{z-t} \geq \frac{(\abs{\phi(z)} - 1)^2}{\abs{\phi(z)}}
 \]
 or
 \[
  \abs{\phi(z)} - 1 \leq \abs{\phi(z)}^{1/2}\abs{z-t}^{1/2}.
 \]
 As the Green's function $g(z) = \log\, \abs{\phi(z)}\geq 0$, this yields
 \[
  g(z) \leq e^{g(z)/2} - e^{-g(z)/2} \leq \dist(z,K)^{1/2},
 \]
 which gives the desired result.
\end{proof}
We now prove an auxiliary lemma which will help us bound the error involved in approximating our discrete measures by appropriately regularized measures (compare \cite[Lemma 2.9]{FRL}):
\begin{lemma}\label{lemma:modulus-of-cont}
 If $c$ is an algebraic number such that $h_{\bM_0}(c)=0$, then $(\mu_0,[c])=0$ and
 $
  0\leq -(\mu_0, [c]_\ep) \leq \sqrt{\ep}.
 $
 Likewise, if $c$ is an algebraic number such that $h_{\bM_1}(c)=0$, then $(\mu_1,[c])=0$ and 
 $
  0\leq -(\mu_1, [c]_\ep) \leq \sqrt{\ep}.
 $
\end{lemma}
\begin{proof}
 We will prove the case for $\bM_0$, the case for $\bM_1$ being identical, \emph{mutatis mutandis}. The canonical height vanishing is equivalent to $z=0$ being preperiodic for $f_c(z) = z^2 + c$. In particular, it follows that $c$ is an algebraic integer, that $G_\bQ c\subset M_0$, and therefore 
 \[
  (\mu_0,[c]) = \int -\log \abs{x-y}\,d\mu_0(x)\,d[c](y) = \int -g(y)\,d[c](y) = 0,
 \]
 where $g(z)=g_{M_0}(z,\infty) = \int \log\abs{z-w}\,d\mu_0(w)$ denotes the Green's function of $M_0$ with respect to infinity. Further, as the support of $[c]_\ep$ is never more than a distance of $\ep$ from $M_0$, and as $M_0$ (and $M_1$) are known to be compact sets of capacity $1$ with connected complements, it follows from Lemma \ref{lemma:loewner} that 
 \[
 (\mu_0,[c]_\ep) = \int -g(y)\,d[c]_{\ep}(y)
 \]
 satisfies $0 \geq (\mu_0,[c])\geq -\sqrt{\ep}$, which proves the desired inequality.
 \end{proof}
\begin{lemma}\label{lemma:arch-2}
 Let $F\subset \bC$ be a finite set and $\ep>0$. Then 
\[
 ([F]_\ep , [F]_\ep) \leq ([F],[F]) - \frac{\log \ep}{\abs{F}}.
\]
\end{lemma}
\noindent This lemma improves on \cite[Lemma 2.10]{FRL} as the term $C/\abs{F}$ on the right hand side is removed.
\begin{proof}
 We follow the same method as in the proof of \cite[Lemma 2.10]{FRL}. We note that for $\ep>0$ and two points $z\neq z'\in\bC$,
\begin{align*}
 -(\delta_{z,\ep},\delta_{z',\ep}) &= \int_0^1\int_0^1 \log\abs{z + \ep\cdot e^{2\pi it} - (z' + \ep\cdot e^{2\pi is})}\,dt\,ds\\
 & = \int_0^1\max\{ \log\abs{z- (z' + \ep\cdot e^{2\pi is})}, \log \ep\}\,ds\\
 &\geq \max\left\{  \int_0^1\log\abs{z- (z' + \ep\cdot e^{2\pi is})}\,ds, \log \ep\right\}\\
 &\geq \max\{\log\abs{z-z'},\log \ep\}\geq \log\abs{z-z'} = -(\delta_z,\delta_{z'})
\end{align*}
so for each $z\neq z'$, we have $(\delta_z,\delta_{z'})\geq (\delta_{z,\ep},\delta_{z',\ep})$. On the other hand, we have what is essentially the Robin constant of the disc of radius $\ep$:
\[
 (\delta_{z,\ep},\delta_{z,\ep}) = - \log \ep.
\]
Thus
\begin{align*}
 ([F]_\ep, [F]_\ep) &= \frac{1}{\abs{F}^2} \sum_{\substack{ z,z'\in F\\ z\neq z'}} (\delta_{z,\ep},\delta_{z',\ep}) + \frac{1}{\abs{F}^2} \sum_{z\in F} (\delta_{z,\ep}, \delta_{z,\ep}) \\
 &\leq \frac{1}{\abs{F}^2} \sum_{\substack{ z,z'\in F\\ z\neq z'}} (\delta_{z},\delta_{z'}) + \frac{1}{\abs{F}^2}\cdot \abs{F} \cdot(-\log \ep)\\
 &= ([F],[F]) - \frac{\log \ep}{\abs{F}}.\qedhere
\end{align*}
\end{proof}

\begin{prop}\label{prop:main-ub}
 Let $c\in\Qbar$ be an algebraic integer of degree $d = [\bQ(c):\bQ]$ and $\ep>0$ be fixed. Suppose that $0,1$ are both preperiodic for $f_c(z)=z^2+c$. Then
 \begin{equation}\label{eqn:mu0-mu1-triangle-ineq}
  d_\infty(\mu_0,\mu_1) \leq 2\left( -\frac{1}{d^2} \log \frac{d^d}{d!} \lfrac{\pi}{4}^{d/2} + 2\sqrt{\ep} + \frac{\log 1/\ep}{d} \right)^{1/2}.
 \end{equation}
\end{prop}
\begin{proof}
 As $[c]_\ep$ admits a continuous potential, we can apply the triangle inequality to obtain:
\[
 d_\infty(\mu_0,\mu_1) \leq d_\infty(\mu_0,[c]_\ep) + d_\infty(\mu_1,[c]_\ep)
\]
Now
\begin{align*}
 d_\infty(\mu_0,[c]_\ep) &= (\mu_0 - [c]_\ep, \mu_0 - [c]_\ep)^{1/2}\\
 &= \left[ (\mu_0,\mu_0) - 2 (\mu_0,[c]_\ep) + ([c]_\ep,[c]_\ep)\right]^{1/2}\\
 &\leq \left[-2(\mu_0,[c]_\ep) + ([c],[c]) + \frac{\log 1/\ep}{d}\right]^{1/2}
\end{align*}
where we have applied Lemma \ref{lemma:arch-2}, and used the fact that 
\[
 (\mu_0,\mu_0) = -\log \ON{cap}(M_0) = 0.
\]
As $c$ is an algebraic integer of degree $d=[\bQ(c):\bQ]$, it is well known that its discriminant is a rational integer which is at least as large as
$
 \frac{d^d}{d!} \lfrac{\pi}{4}^{d/2}
$
and thus 
\[
([c],[c])=\frac{1}{d^2}\sum_{1\leq i\neq j\leq d} -\logabs{c_i-c_j }\leq -\frac{1}{d^2} \log \frac{d^d}{d!} \lfrac{\pi}{4}^{d/2}.
\]
By applying Lemma \ref{lemma:modulus-of-cont} we now obtain:
\begin{equation}\label{eqn:mu0-c-lb}
 d_\infty(\mu_0,[c]_\ep) \leq \left(-\frac{1}{d^2} \log \frac{d^d}{d!} \lfrac{\pi}{4}^{d/2}+ 2\sqrt\ep + \frac{\log 1/\ep}{d}\right)^{1/2}
\end{equation}
The same argument applies for $\mu_1$ and yields the desired result.
\end{proof}

\subsection{Lower bound on the mutual energy of $M_0$ and $M_1$}\label{sec:mu0-mu1-lb}
It now remains to find a lower bound for the distance $d_\infty(\mu_0,\mu_1)$. Our result is the following:
\begin{prop}\label{prop:main-lb}
 The mutual energy distance of the equilibrium measures of the Mandelbrot sets $M_0$ and $M_1$ satisfies $d_\infty(\mu_0,\mu_1)\geq 0.623482\ldots$
\end{prop}
In order to estimate this quantity, we will choose algebraic numbers whose Galois conjugates well-approximate the equilibrium distributions of $M_0$ and $M_1$. One can check that the equation $f_c^{11}(1)=1$ yields 1024 solutions (counting multiplicity) of $c$ for which $z=1$ is periodic; namely, $c=0$ and the 1023 roots of the irreducible polynomial
\begin{multline*}
F(c) = 2047 + 2075647 c + 1393985534 c^2 + 697735695867 c^3 + 
 277762348369394 c^4 \\ + 91636064921989590 c^5 + \cdots + +1177856 c^{1021} + 1536 c^{1022} + c^{1023}
\end{multline*}
Let $\alpha$ denote a root of $F(c)$, and for $\ep>0$ let $[\alpha]_\ep$ denote the regularized Borel probability measure on $\bC$ as defined above in Section \ref{sec:meas-reg} (this measure is the same regardless of the choice of root). Likewise, one can check that the equation $f_c^{11}(0)=0$ has solutions $c=0$ and the roots of the irreducible polynomial
\[
 G(c) = 1 + c + 2 c^2 + 5 c^3 + 14 c^4 + 42 c^5 + \cdots + 130816 c^{1021} + 512 c^{1022} + c^{1023}.
\]
Let $\beta$ be a root of $G(c)$, and let $[\beta]_\ep$ denote the regularized measure associated to it.

Applying the triangle inequality for the mutual energy metric $d$ from Theorem \ref{thm:triangle-ineq} (in fact, the triangle inequality at the archimedean place, although we can consider each measure to have the trivial measure $\lambda_p$ at finite rational primes $p$ if desired; see Remark \ref{rmk:local-triangle-ineq}) we obtain:
\begin{equation}\label{eqn:mu0-mu1-lb}
 d_\infty(\mu_0,\mu_1)\geq d_\infty([\al]_\ep,[\beta]_\ep) - d_\infty(\mu_0,[\al]_\ep) - d_\infty(\mu_1,[\beta]_\ep).
\end{equation}
We will prove Proposition \ref{prop:main-lb} by bounding the terms in this sum for our choice of $\alpha$ and $\beta$. We begin by expanding:
\begin{equation}\label{eqn:Ae-Be-lb}
d_\infty([\al]_\ep,[\beta]_\ep) = \sqrt{([\al]_\ep,[\al]_\ep)_\infty - 2([\al]_\ep,[\beta]_\ep)_\infty + ([\beta]_\ep,[\beta]_\ep)_\infty} 
\end{equation}
We let $d=[\bQ(\al):\bQ]=[\bQ(\beta):\bQ]=1023$ and choose $\ep = 1/d^2 = 1/1023^2$ throughout in the following computations.
\begin{lemma}
 For $\al$ as chosen above, $([\al]_\ep,[\al]_\ep)_\infty \geq 0.00514961\ldots$
\end{lemma}
\begin{proof}
 Let us label the Galois conjugates of $\alpha$ as $\alpha_1,\ldots, \alpha_d$. Then
 \begin{multline*}
  ([\al]_\ep,[\al]_\ep)_\infty = \frac{1}{d^2} \sum_{1\leq i \neq j \leq d} \int -\log\,\abs{x-y}\,d\delta_{\al_i,\ep}(x)\,d\delta_{\al_j,\ep}(y)\\
  + \frac{1}{d^2} \sum_{1\leq i \leq d}  \int -\log\,\abs{x-y}\,d\delta_{\al_i,\ep}(x)\,d\delta_{\al_i,\ep}(y)
 \end{multline*}
where we recall from the definitions above that $\delta_{\al_i,\ep}$ is the normalized Lebesgue measure of the circle $\{\abs{z-\al_i} = \ep\}$, which we recognize as the equilibrium measure of the disc $D(\alpha_i,\ep)$. As is well-known, the terms in the second sum are equal to negation of the logarithmic capacity (with respect to $\infty$) of the discs:
\[
 \int -\log\,\abs{x-y}\,d\delta_{\al_i,\ep}(x)\,d\delta_{\al_i,\ep}(y) = -\log \ep\quad\text{for each}\quad 1\leq i\leq d.
\]
Now we examine the terms in the first sum, that is, for which $i\neq j$. If $\abs{\alpha_i - \alpha_j} > 2\ep$, then the discs $D(\al_i,\ep)$ and $D(\al_j,\ep)$ are disjoint, and by the harmonicity of the potential function outside the disc,
\[
 \int -\log\,\abs{x-y}\,d\delta_{\al_i,\ep}(x)\,d\delta_{\al_j,\ep}(y) = -\log\,\abs{\al_i-\al_j}
\]
for such terms. On the other hand, if $i\neq j$ but $\abs{\alpha_i - \alpha_j} \leq 2\ep$, then $\abs{x-y}\leq 4\ep$ for all $x\in D(\al_i,\ep)$ and $y\in D(\al_j,\ep)$, and thus $-\log\,\abs{x-y}\geq -\log\,(4\ep)$, so
\[
 \int -\log\,\abs{x-y}\,d\delta_{\al_i,\ep}(x)\,d\delta_{\al_j,\ep}(y)\geq -\log\,(4\ep)
\]
in this case. Computing the sums above with these estimates for our chosen $\al$ results in the desired bound.
\end{proof}
Repeating the same estimates above for the conjugates of $\beta$ now yields:
\begin{lemma}
 For $\beta$ as chosen above, $([\beta]_\ep,[\beta]_\ep)_\infty \geq 0.00677490\ldots$
\end{lemma}
We now must estimate the middle term.
\begin{lemma}
 For $\al,\beta$ as chosen above, we have $-2([\al]_\ep,[\beta]_\ep)_\infty\geq 0.630005\ldots$ 
\end{lemma}
\begin{proof}
 We begin by expanding our expression:
 \[
  -2([\al]_\ep,[\beta]_\ep)_\infty = \frac{2}{d^2} \sum_{1\leq i \neq j \leq d} \int \log\,\abs{x-y}\,d\delta_{\al_i,\ep}(x)\,d\delta_{\beta_j,\ep}(y).
 \]
 We again break up the terms in the sum based on the proximity of $\al_i$ and $\beta_j$. Suppose first that $\abs{\al_i-\beta_j}> 2\ep$. Then the discs $D(\al_i,\ep)$ and $D(\beta_j,\ep)$ are disjoint, so as in the previous lemma's proof, we have
 \[
  \int \log\,\abs{x-y}\,d\delta_{\al_i,\ep}(x)\,d\delta_{\beta_j,\ep}(y) = \log\,\abs{\al_i-\beta_j}.
 \]
 Now, the potential function of the disc $D(\beta_j,\ep)$ satisfies:
 \[
  -U^{\delta_{\beta_j,\ep}}(x) = \int \log\,\abs{x-y}\,d\delta_{\beta_j,\ep}(y) \geq \log\,\ON{cap}(D(\beta_j,\ep)) = \log \ep,
 \]
 so for the terms with $\abs{\al_i-\beta_j}\leq 2\ep$, we use instead the estimate that
 \[
  \int \log\,\abs{x-y}\,d\delta_{\al_i,\ep}(x)\,d\delta_{\beta_j,\ep}(y) = \int -U^{\delta_{\beta_j,\ep}}(x)\,d\delta_{\al_i,\ep}(x)\geq \log \ep.
 \]
 With these two estimates, we obtain the desired bound.
\end{proof}
We are now ready to prove Proposition \ref{prop:main-lb}.
\begin{proof}[Proof of Proposition \ref{prop:main-lb}]
 As $\alpha,\beta$ are algebraic integers, all conjugates of $\al$ are contained in $M_0$, and all conjugates of $\beta$ are contained in $M_1$, by the same argument as we used above to derive equation \eqref{eqn:mu0-c-lb}, we have
 \[
  d_\infty(\mu_0,[\al]_\ep) \leq \left( ([\al],[\al]) + 2\sqrt\ep + \frac{\log 1/\ep}{d}\right)^{1/2}
 \]
 and 
 \[
 d_\infty(\mu_1,[\beta]_\ep) \leq \left( ([\beta],[\beta]) + 2\sqrt\ep + \frac{\log 1/\ep}{d}\right)^{1/2}.
 \]
 where 
 \[
  ([\al],[\al]) = \frac{1}{d^2} \sum_{1\leq i \neq \leq d } -\logabs{\al_i - \al_j} = -0.00839974\ldots 
 \]
 and likewise $([\beta],[\beta]) =  -0.00677444\ldots $. Combining this estimate with the lower bound for $d_\infty([\al]_\ep,[\beta]_\ep)$ obtained by using the above lemmas in \eqref{eqn:Ae-Be-lb} in the equation \eqref{eqn:mu0-mu1-lb}, we obtain the desired result.
\end{proof}

\subsection{Proof of Theorem \ref{thm:degree-bound}} 
We are now ready to prove Theorem \ref{thm:degree-bound}.
\begin{proof}[Proof of Theorem \ref{thm:degree-bound}]
 Suppose $c$ is an algebraic integer in $S_{0,1}$ of degree $d=[\bQ(c):\bQ]$. We combine Propositions \ref{prop:main-ub} and \ref{prop:main-lb} to obtain
\begin{equation}
 0.566325\ldots\leq d_\infty(\mu_0,\mu_1) \leq 2 \left(-\frac{1}{d^2} \log \frac{d^d}{d!} \lfrac{\pi}{4}^{d/2}  + 2\sqrt{\ep} + \frac{\log 1/\ep}{d}\right)^{1/2}.
\end{equation}
Taking $\ep=1/d^2$, we immediately see that we must have $d \leq 108$ or else the preceding inequality is violated.
\end{proof}

\subsection{Effective bounds for different choices of initial values}
Although the problem of determining the set $S_{0,1}$ is of particular interest given the known overlap of the Mandelbrot sets $M_0$ and $M_1$, the same techniques used above can be used to give bounds on the degree of elements in the sets $S_{a,b}$ for $a,b$ other rational integers. In each case here, we choose

\bibliographystyle{abbrv}
\bibliography{bib}

\end{document}